\documentclass[12pt]{amsart}
\usepackage{enumerate,amsmath,amssymb,bm}
\usepackage{url}
\usepackage{xspace}

\usepackage[margin=1in]{geometry}

\DeclareMathOperator{\pnt}{\raise 0.5mm \hbox{\large\textbf{.}}}

\newcommand{\note}[2][ ]{}

\newtheorem{theorem}{Theorem}

\newtheorem{problem}[theorem]{Problem}

\theoremstyle{definition}
\newtheorem{remark}[theorem]{Remark}

\newtheorem{question}[theorem]{Question}

\usepackage[pdfauthor={Fabrizio Zanello},
            pdftitle={Three more proofs of two congruences for a special class of partitions},
            pdfsubject={enumerative combinatorics},
            pdfkeywords={regular partitions, parity},
            pdfproducer={Latex with hyperref},
            pdfcreator={latex->dvips->ps2pdf},
            pdfpagemode=UseNone,
            bookmarksopen=false,
            bookmarksnumbered=true]{hyperref}

\begin{document}
\title[Three more proofs for Merca's partition function]{Three more proofs of two congruences\\for Merca's partition function}
\author{Fabrizio Zanello} \address{Department of Mathematical Sciences\\ Michigan Tech\\ Houghton, MI  49931}
\email{zanello@mtu.edu}
\thanks{2020 {\em Mathematics Subject Classification.} Primary: 11P83; Secondary: 05A17, 11P84.\\\indent 
{\em Key words and phrases.} Parity of partitions; eta-quotient; binary $q$-series; Merca's partition function.}

\maketitle

\begin{abstract}
In this note, we provide three new, very short proofs of two interesting congruences for Merca's partition function $a(n)$, which enumerates integer partitions where the odd parts have multiplicity at most 2. These modulo 2 congruences were first shown elementarily by Sellers. We then frame $a(n)$ into the much broader context of eta-quotients, and suggest how to comprehensively describe its parity behavior. In particular, extensive computations suggest that $a(n)$ is odd precisely 25\% of the time.
\end{abstract}
{\ }\\

In an interesting recent article on overpartitions, Merca \cite{mer} considered the parity of the function, $a(n)$, that enumerates integer partitions where no part is 3 (mod 6). Equivalently, $a(n)$ counts the partitions in which all odd parts appear with multiplicity at most 2. Merca proved, among other facts, that $a(n)$ is identically even along the arithmetic progressions $n=4m+2$ and $4m+3$. Answering Merca's call for a more direct and elementary argument, Sellers \cite{sell} found two elegant, additional proofs of his result.

The goal of this brief note is to offer three more, extremely short proofs of Merca's theorem. Further, we place the parity of $a(n)$ into the much broader context of eta-quotients, where the behavior of $a(n)$ (mod 2) is in part determined by that of the 6-regular partitions studied in \cite{KZ2}. Interestingly, extensive computational data suggests that $a(n)$ is odd precisely 25\% of the time (i.e., with natural density $1/4$). Equivalently, under our \lq \lq master conjecture'' on eta-quotients (\cite[Conjecture 4]{KZ}), no nonconstant arithmetic progression appears to exist where $a(n)$ is identically even, outside of the two discovered by Merca.

We say that two series, $\sum b_nq^n$ and $\sum c_nq^n$,  are \emph{congruent modulo 2}, and write $\sum b_nq^n \equiv_2 \sum c_nq^n$, if $b_n \equiv c_n$ (mod 2) for all $n$. Set $f_t=f_t(q)= \prod_{i=1}^\infty (1-q^{ti})$. It is easy to see that the generating function of $a(n)$ satisfies
$$\sum_{n\ge 0}a(n)q^n=\frac{f_3}{f_1f_6}$$
(see, e.g., \cite[Identity (2)]{sell}). Since $f_6\equiv_2 f_3^2$, it follows that
$$\sum_{n\ge 0}a(n)q^n\equiv_2 \frac{1}{f_1f_3}.$$

We also record the following identity that characterizes the parity of 3-core partitions (\cite[Corollary 1]{rob}):
\begin{equation}\label{33}
\frac{f_3^3}{f_1}\equiv_2 \sum_{n\in \mathbb Z}q^{n(3n-2)}.
\end{equation}
We are now ready for our three proofs of Merca's result.

\begin{theorem}[\cite{mer}]
$a(4m+2)$ and $a(4m+3)$ are identically zero (mod 2).
\end{theorem}

\begin{proof}[First proof.]
We employ the modulo 2 identity
\begin{equation}\label{99}
f_1^3\equiv_2 f_3 +qf_9^3,
\end{equation}
which is a simple consequence of the following two classical formulas on pentagonal and triangular numbers, respectively:
$$f_1\equiv_2 \sum_{n\in \mathbb Z}q^{n(3n-1)/2};{\ }{\ }{\ }f_1^3\equiv_2 \sum_{n\ge 0}q^{\binom{n+1}{2}}.$$ 
Dividing across (\ref{99}) by $f_1^4f_3$ yields:
$$\frac{1}{f_1f_3}\equiv_2 \frac{1}{f_1^4}+q\frac{f_9^3}{f_3}\cdot \frac{1}{f_1^4}.$$
By (\ref{33}) with $q^3$ in lieu of $q$, the last displayed formula is
$$\frac{1}{f_1^4}\left( 1+q\frac{f_9^3}{f_3}\right) \equiv_2 \frac{1}{f_1^4}\left( 1+\sum_{n\in \mathbb Z} q^{1+3n(3n-2)}\right) \equiv_2 \frac{1}{f_1^4}\left( 1+\sum_{n\in \mathbb Z} q^{(3n-1)^2}\right).$$
The conclusion now follows since $f_1^4 \equiv_2 f_4$, and all integer squares are 0 or 1 (mod 4).
\end{proof}

\begin{proof}[Second proof.] 
By (\ref{33}), we have:
$$\frac{1}{f_1f_3}\equiv_2 \frac{1}{f_3^4}\cdot \frac{f_3^3}{f_1}\equiv_2 \frac{1}{f_3^4}\cdot \sum_{n\in \mathbb Z}q^{n(3n-2)}.$$
The proof is complete by observing that $f_3^4 \equiv_2 f_{12}$, and that $n(3n-2)$ is never 2 or 3 (mod 4).
\end{proof}

Note that, in a sense, our second proof greatly simplifies the idea behind Section 3 of \cite{sell}.

\begin{proof}[Third proof.] 
We make use of the following fact:
\begin{equation}\label{12}
f_1^3f_3^3\equiv_2 f_1^{12} +qf_3^{12}
\end{equation}
(see \cite[Theorem 2.1]{HS} and \cite[Formula (12)]{JKZ}).
We divide across (\ref{12}) by $f_1^4f_3^4,$ to obtain:
\begin{equation}\label{44}
\frac{1}{f_1f_3}\equiv_2 \frac{f_1^8}{f_3^4}+q\frac{f_3^8}{f_1^4}\equiv_2 \frac{f_8}{f_{12}}+q\frac{f_{24}}{f_4}.
\end{equation}
The result again immediately follows.
\end{proof}

\begin{remark}\label{r1}
Using the right side of (\ref{44}), we easily see that the parity of $a(n)$ along the arithmetic progression $4m+1$ is determined by the parity of 6-regular partitions, whose generating function is $f_6/f_1$. We conjectured in \cite[Conjecture 7]{KZ2} that the 6-regular partition function is odd precisely 50\% of the time, and in fact, that it is equidistributed modulo 2 along any nonconstant arithmetic progression. See \cite{KZ2} for more details.
\end{remark}

We conclude with some suggestions for future work, in order to comprehensively describe the behavior modulo 2 of Merca's partition function.

\begin{problem}\label{p1}
Show that:
\begin{itemize}
\item[ (1)] $a(n)$ is odd precisely 25\% of the time. In particular, by (\ref{44}), the goal is to prove that both the 6-regular partition function and the Fourier coefficients of the eta-quotient $f_2/f_3$ are equidistributed modulo 2.
\item[ (2)] There exists no nonconstant arithmetic progression, disjoint from $4m+2$ and $4m+3$, where $a(n)$ is identically even.
\end{itemize}
\end{problem}

\begin{remark}\label{r2}
It can be shown that, under our \lq \lq master conjecture'' on the parity of eta-quotients (\cite[Conjecture 4]{KZ}), the statements of parts (1) and (2) of Problem \ref{p1} are equivalent.
\end{remark}

In fact, we ask the following bolder question, which would further illuminate the behavior of $a(n)$.

\begin{question}\label{q1}
 Is $a(n)$ equidistributed modulo 2 along \emph{any} nonconstant arithmetic progression where it is not identically even?
\end{question}

\section*{Acknowledgements} We thank William Keith for useful computations. This work was partially supported by a Simons Foundation grant (\#630401).


\begin{thebibliography}{99}

\bibitem{HS} M. Hirschhorn and J. Sellers: \emph{On recent congruence results of Andrews and Paule for broken $k$-diamonds}, Bull. Austral. Math. Soc. \textbf{75} (2007), 121--126.

\bibitem{JKZ} S. Judge, W.J. Keith, and F. Zanello: \emph{On the density of the odd values of the partition function}, Ann. Comb. \textbf{22} (2018), no. 3, 583--600.

\bibitem{KZ} W.J. Keith and F. Zanello: \emph{Parity of the coefficients of certain eta-quotients}, J. Number Theory \textbf{235} (2022), 275--304.

\bibitem{KZ2} W.J. Keith and F. Zanello: \emph{Parity of the coefficients of certain eta-quotients, II: The case of even-regular partitions}, J. Number Theory \textbf{251} (2023), 84--101.

\bibitem{mer} M. Merca: \emph{Overpartitions in terms of 2-adic valuation}, Aequat. Math. \textbf{99} (2025), 1153--1173. 

\bibitem{rob} N. Robbins: \emph{On $t$-core partitions}, Fibonacci Quart. \textbf{38} (2000), no. 1, 39--48.

\bibitem{sell} J. Sellers: \emph{Elementary proofs of two congruences for partitions with odd parts repeated at most twice}, preprint (arXiv:2409.12321; version v2 of August 8, 2025). The Mathematics Student (Indian Math. Soc.), to appear.

\end{thebibliography}
\end{document}